\documentclass{amsart}
\usepackage[lite]{amsrefs}
\usepackage{amssymb,amsmath,amsthm,amscd,latexsym}
\usepackage{amssymb}
\usepackage{amsfonts}
\usepackage[mathscr]{eucal}
\usepackage{color}
\usepackage{graphicx}
\usepackage{float}
\usepackage[all,cmtip]{xy}
\usepackage{tikz}

\newcommand{\FF}{\mathbb{F}}

\newcommand{\ds}{\displaystyle}

\usepackage{color}

\newlength{\NT}

\begingroup
\newtheorem{thrm}{Theorem\hskip 1mm}[section]

\endgroup

\begin{document}

\title[]{New Extremal Binary Self-dual Codes from block circulant matrices and block quadratic residue circulant matrices}

\author{J Gildea, A Kaya, R Taylor, A Tylyshchak, B Yildiz}

\address{Department of Mathematics\\
Faculty of Science and Engineering\\
University of Chester\\
England}
\email{j.gildea@chester.ac.uk}

\address{Department of Mathematics Education, Sampoerna University\\
12780, Jakarta, Indonesia}
\email{nabidin@gmail.com}

\address{Department of Mathematics\\
Faculty of Science and Engineering\\
University of Chester\\
England}
\email{rhian.taylor@chester.ac.uk}

\address{Department of Algebra\\
Uzhgorod National University\\
Uzhgorod, Ukraine}
\email{alxtlk@bigmir.net}

\address{Department of Mathematics \& Statistics\\
Northern Arizona University  \\
Flagstaff, AZ 86001 \\
USA}
\email{Bahattin.Yildiz​@nau.edu}

\keywords{self-dual codes, codes over rings, quadratic double circulant codes}
\subjclass{94B05,15B33}

\begin{abstract}
In this paper, we construct self-dual codes from a construction that involves
both block circulant matrices and block quadratic residue circulant matrices. We provide
conditions when this construction can yield self-dual codes. We construct
self-dual codes of various lengths over $\FF_2$ and $\FF_2+u\FF_2$. Using extensions, neighbours
and sequences of neighbours, we construct many new self-dual codes. In particular, we construct
one new self-dual code of length $66$ and $51$ new self-dual codes of length $68$.
\end{abstract}

\maketitle

\section{Introduction}

Self-dual codes are a class of linear block codes that have been extensively studied in recent history. One of the most famous and extensively used constructions, used to construct self-dual codes, is the double circulant construction. It involves considering a generator matrix of the form $(I|A)$ where $A$ is a circulant matrix. In 2002, Gaborit (\cite{gaborit}) introduced the notion of a quadratic residue circulant matrix. Let $R$ be a finite commutative Frobenius ring of characteristic $2$ and $p$ be prime. Let $\gamma_i \in R$, $A$ be a $p \times p$ circulant matrix, $Q_r(a,b,c)$ be the $p \times p$ circulant matrix with three free variables, obtained through the quadratic residues and non-residues modulo $p$. Thus, the first row of $\overline{r} = (r_0,r_1, \dots, r_{p-1})$ of $Q_p(a,b,c)$ is determined by the following rule:
\[
\begin{split}
r_i=
\begin{cases}
a & \text{if}\;i=0\\
b & \text{if}\;i \; \text{is a quadratic residue modulo} \; p\\
c & \text{if}\;i \; \text{is a quadratic non-residue modulo} \; p.\\
\end{cases}\\[0.1in]
\end{split}
\]

In \cite{gaborit}, Gaborit considered constructing self-dual codes from generator matrices of the form  $(I|Q_p(a,b,c))$ and

$$\left(
\begin{array}{c|ccc|c|ccc}
\gamma_1   & \gamma_2 & \cdots & \gamma_2  & \gamma_3 &  \gamma_4 & \cdots  &  \gamma_4 \\ \hline
\gamma_2   &  &  &  & \gamma_4 &   &   &   \\
\vdots   &  & I &  & \vdots &   & Q_p(a,b,c)  &   \\
\gamma_2   &  &  &  & \gamma_4 &   &   &  \\
\end{array}  \right).$$

\noindent In \cite{QR}, these techniques were extended to constructing self-dual codes from generator matrices of the form
$(Q_p(a,b,c)|A)$ and
$$\left(
\begin{array}{c|ccc|c|ccc}
\gamma_1   & \gamma_2 & \cdots & \gamma_2  & \gamma_3 &  \gamma_4 & \cdots  &  \gamma_4 \\ \hline
\gamma_2   &  &  &  & \gamma_4 &   &   &   \\
\vdots   &  & Q_p(a,b,c) &  & \vdots &   & A  &   \\
\gamma_2   &  &  &  & \gamma_4 &   &   &  \\
\end{array}  \right),$$
where $A$ is a $p \times p$ circulant matrix. In this article we consider constructing self-dual codes from generator matrices of the
form

\[
\left(
\begin{array}{ccc|ccc}
Q_0& Q_1 & Q_2 & A_0 & A_1 & A_2\\
Q_2  & Q_0& Q_1  &A_2 &  A_0 & A_1 \\
Q_1  & Q_2  & Q_0 &A_1 & A_2 &  A_0   \\
\end{array}
\right)
\]

where $Q_i$ are quadratic residue circulant matrices and $A_i$ are $p \times p$ circulant matrices.\\

 Section 2 of this article contains a brief introduction to
self-dual codes. We discuss some important properties of quadratic residue circulant matrices in section 3.  In section 4, we describe the construction itself. We provide theoretical results that establish certain conditions when this construction yields self-dual codes. In section 5,
we apply the construction to find many known and unknown self-dual codes that had not been previously constructed. We conclude with listing the newly constructed codes and a suggestion for future work.

\section{Preliminaries}

Throughout this paper, $R$ will denote a commutative Frobenius ring of characteristic 2. A
code $C$ of length $n$ over $R$ is an $R$-submodule of $R^{n}$. Elements of the code $C$ are called codewords of $C$. Let
$x=\left( x_{1},x_{2},\ldots ,x_{n}\right)\in R^n$ and $y=\left( y_{1},y_{2},\ldots
,y_{n}\right)\in R^n $. Define the Euclidean inner product between $x$ and $y$ as $\left\langle
x,y\right\rangle _{E}=\sum x_{i}y_{i}$. The dual $C^{\bot }$ of the code $C$
is defined as
\begin{equation*}
C^{\bot }=\left\{ x\in R^{n} \mid \left\langle x,y\right\rangle
_{E}=0\text{ for all }y\in C\right\} .
\end{equation*}%

 If $C=C^{\bot }$, we say that $C$ is self-dual. For binary codes, a self-dual code where all weights are congruent to $0 \pmod{4}$ is said to be Type~II and a self-dual binary code is said to be Type~I otherwise. The bounds on the minimum distances for self-dual codes are given in \cite{Rains} and are as follows:

\begin{thrm}
$($\cite{Rains}$)$ Let $d_{I}(n)$ and $d_{II}(n)$ be the minimum distances of
a Type I and Type II binary code of length $n$, respectively. Then
\begin{equation*}
d_{II}(n)\leq 4\lfloor \frac{n}{24}\rfloor +4
\end{equation*}%
and
\begin{equation*}
d_{I}(n)\leq \left\{
\begin{array}{ll}
4\lfloor \frac{n}{24}\rfloor +4 & \text{if $n\not\equiv 22\pmod{24}$} \\
4\lfloor \frac{n}{24}\rfloor +6 & \text{if $n\equiv 22\pmod{24}$.}%
\end{array}%
\right.
\end{equation*}
\end{thrm}

Self-dual codes that meet these bounds are called \textit{extremal}.\\

 Although, the theoretical result in this article is based around commutative Frobenius rings of characteristic 2,
 all the computational results are based on the rings $\FF_2$ and $\FF_2+u\FF_2$. Now, $\FF_2+u\FF_2 := \FF_2 [X]/(X^2)$, where $u$ satisfies $u^2=0$. Thus, the elements of the ring are $0,1,u$ and $1+u$, where $1$ and $1+u$ are the units of $\FF_2+u\FF_2$. We also define the Gray map $\phi$ from $\FF_2+u\FF_2$ to $\FF_2^2$ given by $\phi(a+bu)=(b,a+b)$ where $a,b \in \FF_2$.\\

  The next result, introduced in \cite{Kim}, will be implemented throughout this article.\\

\begin{thrm}\label{extension}
Let $C$ be a binary self-dual code of length $2n$, $G=(r_{i})$ be an $n\times 2n$ generator matrix for $C$,
where $r_{i}$ is the $i$-th row of $G$, $1\leq i\leq n$. Let $X$ be a vector in $\FF_2^{2n}$ with $\left\langle
X,X\right\rangle =1$. Let $y_{i}=\left\langle r_{i},X\right\rangle $ for $%
1\leq i\leq n$. Then the following matrix%
\begin{equation*}
\left[
\begin{array}{cc|c}
1 & 0 & X \\ \hline
y_{1} & y_{1} & r_{1} \\
\vdots  & \vdots  & \vdots  \\
y_{n} & y_{n} & r_{n}%
\end{array}%
\right] ,
\end{equation*}%
generates a binary self-dual code of length $2n+2$.
\end{thrm}

\bigskip

Two self-dual binary codes of dimension $k$ are said to be neighbours
if their intersection has dimension $k-1$. Let $C$ be a self-dual
code. Let $x\in {\mathbb{F}}_{2}^{n}-C$ then $D=\left\langle \left\langle x\right\rangle ^{\bot }\cap
C,x\right\rangle $ is a neighbour of $C$. Let $x_0 \in \FF_{2}^{2n}-\mathcal{N}_{(0)}$. In \cite{GNP}, the following formula for constructing the $k$-range neighbour codes was provided:

\[
\mathcal{N}_{(i+1)}=\left\langle \left\langle x_i \right\rangle ^{\bot }\cap \mathcal{N}_{(i)},x_i\right\rangle
\]

\noindent where $\mathcal{N}_{(i+1)}$ is the  neighbour of $\mathcal{N}_{(i)}$ and $x_i \in \FF_{2}^{2n}-\mathcal{N}_{(i)}$.\\

\bigskip

\section{Quadratic Residue Circulant Matrices}

\noindent Let $Q_p(a_i,b_i,c_i)$ be the $i^{th}$-$p \times p$ quadratic circulant matrix, where $a_i,b_i,c_i \in R$ and $p$ is a prime number and $0 \leq i \leq 2$. For the purposes
of this article, we need to evaluate $ Q_p(a_i,b_i,c_i)Q_p(a_j,b_j,c_j)^T$. From \cite{gaborit},
we can clearly see that $ Q_p(a_i,b_i,c_i)Q_p(a_i,b_i,c_i)^T$
\[=\begin{cases}
Q_p(a_i^2,b_i^2+k(b_i^2+c_i^2),c_i^2+k(b_i^2+c_i^2)) & \text{if}\;p=4k+1\\
Q_p(a_i^2+b_i^2+c_i^2,a_ib_i+a_ic_i+b_ic_i+(b_i^2+c_i^2)k,a_ib_i+a_ic_i+b_ic_i+(b_i^2+c_i^2)k)  & \text{if}\;p=4k+3\\
\end{cases}.\]

\noindent We shall now calculate $ Q_p(a_i,b_i,c_i)Q_p(a_j,b_j,c_j)^T$. First we will consider the case when $p=4k+1$ and then the case when $p=4k+3$.

\begin{thrm}\label{T:1} If $p=4k+1$ then $Q_p(a_i,b_i,c_i)Q_p(a_j,b_j,c_j)^T$
\[=Q_p(a_ia_j,a_ib_j+b_ia_j+(k+1)b_ib_j+k(b_ic_j+c_ib_j)+kc_ic_j,a_ic_j+c_ia_j+kb_ib_j+k(b_ic_j+c_ib_j+(k+1)c_ic_j).\]
\end{thrm}
\begin{proof}
Assume that $p=4k+1$. Let $Q=Q_p(0,1,0)$ and $N=Q_p(0,0,1)$, then
\[
\begin{split}
Q_p(a_i,b_i,c_i)Q_p(a_j,b_j,c_j)^T&=(a_iI+b_iQ+c_iN)(a_jI+b_jQ+c_jN)^T\\
&=(a_iI+b_iQ+c_iN)(a_jI+b_jQ^T+c_jN^T)\\
&=a_ia_jI+a_ib_jQ^T+a_ic_jN^T+b_ia_jQ+b_ib_jQQ^T\\
&\;\;\;\;+b_ic_jQN^T+c_ia_jN+c_ib_jNQ^T+c_ic_jNN^T.
\end{split}
\]

\noindent Recall (\cite{gaborit}) that $Q=Q^T$, $N=N^T$, $QQ^T=(k+1)Q+kN$, $QN^T=NQ^T=k(Q+N)$ and $NN^T=kQ+(k+1)N$. Therefore,
\[
\begin{split}
Q_p(a_i,b_i,c_i)Q_p(a_j,b_j,c_j)^T=&a_ia_jI+(a_ib_j+b_ia_j)Q+(a_ic_j+c_ia_j)N+b_ib_j( (k+1)Q+kN)\\
&+(b_ic_i+c_ib_j)(k(Q+N))+c_ic_j(kQ+(k+1)N)\\
=&a_ia_jI+(a_ib_j+b_ia_j)Q+(a_ic_j+c_ia_j)N+b_ib_j(k+1)Q+b_ib_jkN\\
&+(b_ic_i+c_ib_j)kQ+(b_ic_i+c_ib_j)kN+c_ic_jkQ+c_ic_j(k+1)N\\
=&I[a_ia_j]+Q[a_ib_j+b_ia_j+(k+1)b_ib_j+k(b_ic_j+c_ib_j)+kc_ic_j]\\
&+N[a_ic_j+c_ia_j+kb_ib_j+k(b_ic_j+c_ib_j)+(k+1)c_ic_j]
\end{split}
\]
$=Q_p(a_ia_j,a_ib_j+b_ia_j+(k+1)b_ib_j+k(b_ic_j+c_ib_j)+kc_ic_j,a_ic_j+c_ia_j+kb_ib_j+k(b_ic_j+c_ib_j)+(k+1)c_ic_j).$
\end{proof}

\begin{thrm}\label{T:2} If $p=4k+3$ then $Q_p(a_i,b_i,c_i)Q_p(a_j,b_j,c_j)^T$
\[
\begin{split}
=Q_p(&a_ia_j+b_ib_j+c_ic_j,(a_ic_j+b_ia_j)+k(b_ib_j+c_ic_j)+kb_ic_j+(k+1)c_ib_j,\\
&(a_ib_j+c_ia_j)+k(b_ib_j+c_ic_j)+(k+1)b_ic_j+kc_ib_j)
\end{split}.\]
\end{thrm}
\begin{proof}
Assume that $p=4k+3$. Then
\[
\begin{split}
Q_p(a_i,b_i,c_i)Q_p(a_j,b_j,c_j)^T&=a_ia_jI+a_ib_jQ^T+a_ic_jN^T+b_ia_jQ+b_ib_jQQ^T\\
&\;\;\;\;+b_ic_jQN^T+c_ia_jN+c_ib_jNQ^T+c_ic_jNN^T.
\end{split}
\]

\noindent Recall (\cite{gaborit}) that $Q=N^T$, $QQ^T=NN^T=I+kQ+kN$, $QN^T=kQ+(k+1)N$ and $NQ^T=(k+1)Q+kN$. Therefore,

\[
\begin{split}
Q_p(a_i,b_i,c_i)Q_p(a_j,b_j,c_j)^T=&a_ia_jI+(a_ic_j+b_ia_j)Q+(a_ib_j+c_ia_j)N+(b_ib_j+c_ic_j)QQ^T+b_ic_jQN^T+c_ib_jNQ^T\\
=&a_ia_jI+(a_ic_j+b_ia_j)Q+(a_ib_j+c_ia_j)N+(b_ib_j+c_ic_j)(I+kQ+kN)\\
&+b_ic_j(kQ+(k+1)N)+c_ib_j((k+1)Q+kN)\\
&=a_ia_jI+(a_ic_j+b_ia_j)Q+(a_ib_j+c_ia_j)N+(b_ib_j+c_ic_j)I+k(b_ib_j+c_ic_j)Q\\
&+k(b_ib_j+c_ic_j)N+kb_ic_jQ+(k+1)b_ic_jN+(k+1)c_ib_jQ+kc_ib_jN\\
&=I[a_ia_j+b_ib_j+c_ic_j]+Q[(a_ic_j+b_ia_j)+k(b_ib_j+c_ic_j)+kb_ic_j\\
&+(k+1)c_ib_j]+N[(a_ib_j+c_ia_j)+k(b_ib_j+c_ic_j)+(k+1)b_ic_j+kc_ib_j]\\
&=Q_p(a_ia_j+b_ib_j+c_ic_j,(a_ic_j+b_ia_j)+k(b_ib_j+c_ic_j)+kb_ic_j+(k+1)c_ib_j,\\
&(a_ib_j+c_ia_j)+k(b_ib_j+c_ic_j)+(k+1)b_ic_j+kc_ib_j)
\end{split}
\]

\end{proof}

\section{The Construction}

\noindent We shall now describe the main construction itself and provide conditions when this
construction produces self-dual codes. Let $Q_l=Q_p(a_l,b_l,c_l)$. Define the matrix
\[
M=
\left(
\begin{array}{ccc|ccc}
Q_0& Q_1 & Q_2 & A_0 & A_1 & A_2\\
Q_2  & Q_0& Q_1  &A_2 &  A_0 & A_1 \\
Q_1  & Q_2  & Q_0 &A_1 & A_2 &  A_0   \\
\end{array}
\right)
\]
\noindent and let $\mathcal{C}$ be the linear code of length $6p$ generated by the matrix $M$, where $A_i$ are $p \times p$ circulant matrices over $R$. Let $CIRC(A_1,\ldots,A_n)$ be the block circulant matrix where the first row of block matrices are $A_1,\ldots,A_n$ and $a_{[l]_3}=a_{ (l \mod 3 )}$, then

\[ MM^T=CIRC \left( \sum_{i=0}^2(Q_iQ_i^T+A_iA_i^T), \sum_{i=0}^{2} Q_iQ_{[(i+2)]_3}^T+A_iA_{[(i+2)]_3}^T , \left(  \sum_{i=0}^{2} Q_iQ_{[(i+2)]_3}^T+A_iA_{[(i+2)]_3}^T \right)^T\right).\]


\noindent Clearly, $C$ is self-orthogonal if and only $\ds{\sum_{i=0}^2A_iA_i^T=\sum_{i=0}^2Q_iQ_i^T}$ and $\ds{\sum_{i=1}^{3}}A_iA_{[(i+2)]_3}^T=\sum_{i=1}^{3} Q_iQ_{[(i+2)]_3}^T$. \newline Using Theorem \ref{T:1}, we can see that
$\ds{\sum_{i=0}^2Q_iQ_i^T}=$
\[
\begin{cases}
Q_p \left( \ds{\sum_{i=0}^2}a_i^2,\sum_{i=0}^2(b_i^2+k(b_i^2+c_i^2)),\sum_{i=0}^2(c_i^2+k(b_i^2+c_i^2))  \right) &  \text{if}\;p=4k+1\\
Q_p \left(\ds{\sum_{i=0}^2}(a_i^2+b_i^2+c_i^2),\sum_{i=0}^2(a_ib_i+a_ic_i+b_ic_i+k(b_i^2+c_i^2),\sum_{i=0}^2(a_ib_i+a_ic_i+b_ic_i+k(b_i^2+c_i^2)   \right) &  \text{if}\;p=4k+3
\end{cases}.
\]

Additionally (by Theorem \ref{T:2}), if $p=4k+1$ then
\[
\begin{split}
\sum_{i=1}^{3} Q_i & Q_{[(i+2)]_3}^T=Q_p\left( \ds{\sum_{i=0}^2}\right. a_ia_{[(i+2)]_3},\sum_{i=0}^2( a_ib_{[(i+2)]_3}+b_ia_{[(i+2)]_3}+(k+1)b_ib_{[(i+2)]_3}+k(b_ic_{[(i+2)]_3}+c_ib_{[(i+2)]_3}\\
 &+kc_ic_{[(i+2)]_3}), \left. \ds{\sum_{i=0}^2} (a_ic_{[(i+2)]_3}+c_ia_{[(i+2)]_3}+kb_ib_{[(i+2)]_3}+k(b_ic_{[(i+2)]_3}+c_ib_{[(i+2)]_3}+(k+1)c_ic_{[(i+2)]_3}) \right)\\
\end{split}
 \]
and if $p=4k+3$ then
\[
\begin{split}
  \sum_{i=1}^{3} Q_iQ_{[(i+2)]_3}^T&=Q_p \left(\ds{\sum_{i=0}^2}\right.(a_ia_{[(i+2)]_3}+b_ib_{[(i+2)]_3}+c_ic_{[(i+2)]_3},\ds{\sum_{i=0}^2}[(a_ic_{[(i+2)]_3}+b_ia_{[(i+2)]_3})+k(b_ib_{[(i+2)]_3}+c_ic_{[(i+2)]_3})\\
  &+kb_ic_{[(i+2)]_3} +(k+1)c_ib_{[(i+2)]_3}],\ds{\sum_{i=0}^2}[(a_ib_{[(i+2)]_3}+c_ia_{[(i+2)]_3})+k(b_ib_{[(i+2)]_3}+c_ic_{[(i+2)]_3})\\
  & +(k+1)b_ic_{[(i+2)]_3}+kc_ib_{[(i+2)]_3})] \Bigg)
\end{split}
\]

\noindent Combining these results, we reach the following:

\begin{thrm}
Assume that $p=4k+1$. Then, $C$ is a self-orthogonal code if and only if the following conditions hold:
\begin{enumerate}
\item $\ds{\sum_{i=0}^2A_iA_i^T=Q_p \left( \ds{\sum_{i=0}^2}a_i^2,\sum_{i=0}^2(b_i^2+k(b_i^2+c_i^2)),\sum_{i=0}^2(c_i^2+k(b_i^2+c_i^2))  \right)}$,
\item
\[
\begin{split}
\sum_{i=1}^{3}& A_iA_{[(i+2)]_3}^T =Q_p\left( \ds{\sum_{i=0}^2}\right. a_ia_{[(i+2)]_3},\sum_{i=0}^2( a_ib_{[(i+2)]_3}+b_ia_{[(i+2)]_3}+(k+1)b_ib_{[(i+2)]_3}+k(b_ic_{[(i+2)]_3}+c_ib_{[(i+2)]_3}\\
 & +kc_ic_{[(i+2)]_3}),\left. \ds{\sum_{i=0}^2} (a_ic_{[(i+2)]_3}+c_ia_{[(i+2)]_3}+kb_ib_{[(i+2)]_3}+k(b_ic_{[(i+2)]_3}+c_ib_{[(i+2)]_3}+(k+1)c_ic_{[(i+2)]_3}) \right).
\end{split}
\]
\end{enumerate}
\end{thrm}

\begin{thrm}
Assume that $p=4k+3$. Then, $C$ is a self-orthogonal code if and only if the following conditions hold:
\begin{enumerate}
\item $\ds{\sum_{i=0}^2A_iA_i^T=Q_p \left(\ds{\sum_{i=0}^2}(a_i^2+b_i^2+c_i^2),\sum_{i=0}^2(a_ib_i+a_ic_i+b_ic_i+k(b_i^2+c_i^2),\sum_{i=0}^2(a_ib_i+a_ic_i+b_ic_i+k(b_i^2+c_i^2)   \right) }$,
\item
\[
\begin{split}
\sum_{i=1}^{3}A_iA_{[(i+2)]_3}^T&=Q_p \left(\ds{\sum_{i=0}^2}\right.(a_ia_{[(i+2)]_3}+b_ib_{[(i+2)]_3}+c_ic_{[(i+2)]_3},\ds{\sum_{i=0}^2}[(a_ic_{[(i+2)]_3}+b_ia_{[(i+2)]_3})+kb_ib_{[(i+2)]_3}\\
  &+kc_ic_{[(i+2)]_3}+kb_ic_{[(i+2)]_3} +(k+1)c_ib_{[(i+2)]_3}],\ds{\sum_{i=0}^2}[(a_ib_{[(i+2)]_3}+c_ia_{[(i+2)]_3})+kb_ib_{[(i+2)]_3}\\
  &+kc_ic_{[(i+2)]_3} +(k+1)b_ic_{[(i+2)]_3}+kc_ib_{[(i+2)]_3})] \Bigg).
\end{split}
\]
\end{enumerate}
\end{thrm}

\begin{thrm}
The matrix $M$ has full rank iff the following conditions hold:
\begin{enumerate}
\item $\ds{ \sum_{i=0}^2 (A_iC_i+A_iD_i )=I_p   }$,
\item $\ds{ \sum_{i=0}^2 (A_iC_{[i+2]_3}+A_iD_{[i+2]_3} )=0_p   }$ and
\item $\ds{ \sum_{i=0}^2 (A_iC_{[i+1]_3}+A_iD_{[i+1]_3} )=0_p   }$
\end{enumerate}
for some $p \times p$ circulant matrices $C_k$ and $D_l$ over $R$.
\end{thrm}
\begin{proof}
Clearly,
\[M= \left(
\begin{array}{c|c}
CIRC(Q_0,Q_1,Q_2)&CIRC(A_0,A_1,A_2)
\end{array}
\right)\]
has full rank iff $MN=I_{3p}$ for some $6p \times 3p$ matrix $N$ over $R$. Let $N'=(n_1,\ldots,n_{6p})^T$ be the first column of $N$, clearly
$M( circ(n_1,\ldots,n_p)^T,\ldots, circ(n_{5p+1},\ldots,n_{6p})^T)^T=(I_p,0_p,0_p,0_p,0_p,0_p)^T$. If $N''=(C_0,C_1,C_2,D_0,D_1,D_2)^T$ is the matrix
that satisfies $MN''=(I_p,0_p,0_p,0_p,0_p,0_p)^T$, then $N$ can take the form

\[ N=\left(
\begin{array}{c}
CIRC(C_0,C_2,C_1)\\
CIRC(D_0,D_2,D_1)
\end{array}
\right)\]
where $C_k$ and $D_l$ are $p \times p$ circulant  matrices over $R$.  Now,
\[ MN =CIRC \left( \sum_{i=0}^2 (A_iC_i+A_iD_i ), \sum_{i=0}^2 (A_iC_{[i+2]_3}+A_iD_{[i+2]_3} ),\sum_{i=0}^2 (A_iC_{[i+1]_3}+A_iD_{[i+1]_3} )   \right)\]

\noindent and $M$ has full rank iff:

\begin{enumerate}
\item $\ds{ \sum_{i=0}^2 (A_iC_i+A_iD_i )=I_p   }$,
\item $\ds{ \sum_{i=0}^2 (A_iC_{[i+2]_3}+A_iD_{[i+2]_3} )=0_p   }$ and
\item $\ds{ \sum_{i=0}^2 (A_iC_{[i+1]_3}+A_iD_{[i+1]_3} )=0_p   }$
\end{enumerate}
\end{proof}

\begin{thrm}
Let $\mathcal{C}$ be self-dual. Then,
\[ \left( \sum_{i=0}^2Q_i \right)B+\left( \sum_{i=0}^2Q_i  \right)^TB'=I_p\]
for some $p \times p$ matrices $B$ and $B'$ over $R$.
\end{thrm}
\begin{proof}
By the previous result,
\begin{enumerate}
\item $\ds{ \sum_{i=0}^2 (A_iC_i+A_iD_i )=I_p   }$,
\item $\ds{ \sum_{i=0}^2 (A_iC_{[i+2]_3}+A_iD_{[i+2]_3} )=0_p   }$ and
\item $\ds{ \sum_{i=0}^2 (A_iC_{[i+1]_3}+A_iD_{[i+1]_3} )=0_p   }$.
\end{enumerate}

\noindent Adding these equations, we obtain that
\[  \left( \sum_{i=0}^2Q_i \right) \left(  \sum_{i=0}^2C_i\right) +\left( \sum_{i=0}^2A_i \right) \left(  \sum_{i=0}^2D_i\right) =I_p.\]

\noindent Let $Q_3=\ds{\sum_{i=0}^2} Q_i$, $A_3=\ds{\sum_{i=0}^2} A_i$, $C_3=\ds{\sum_{i=0}^2} C_i$ and $D_3=\ds{\sum_{i=0}^2} D_i$. Thus,
\[Q_3C_3+A_3D_3=I_p\]
and
\[ (Q_3C_3+A_3D_3)^T=C_3^TQ_3^T+D_3^TA_3^T=Q_3^TC_3^T+A_3^TD_3^T=I_p\]
since circulant matrices commute. Therfore,
\[
\begin{split}
  Q_3C_3+A_3D_3&=Q_3C_3+A_3(Q_3^TC_3^T+A_3^TD_3^T)D_3\\
  &=Q_3C_3+A_3Q_3^TC_3^TD_3+A_3A_3^TD_3^TD_3\\
  &=I_p.
\end{split}
\]

\noindent If $\mathcal{C}$ is self-dual, then $MM^T=0_{3p}$ and
\[
\left(\begin{array}{ccc}I_p&I_p&I_p\end{array}\right)
MM^T
\left(\begin{array}{ccc}I_p&I_p&I_p\end{array}\right)^T
=0_p.\]

\noindent Consequently,
\[
\left(
\begin{array}{cccccc}
Q_3&Q_3&Q_3&A_3&A_3&A_3
\end{array}
\right)
\left(
\begin{array}{cccccc}
Q_3&Q_3&Q_3&A_3&A_3&A_3
\end{array}
\right)^T=0_{p}\mbox{ and }Q_3Q_3^T=A_3A_3^T.
\]
\noindent Finally,
\[
\begin{split}
I_p &= Q_3C_3+A_3Q_3^TC_3^TD_3+A_3A_3^TD_3^TD_3\\
&= Q_3C_3+A_3Q_3^TC_3^TD_3+Q_3Q_3^TD_3^TD_3\\
&=Q_3C_3+Q_3Q_3^TD_3^TD_3+A_3Q_3^TC_3^TD_3\\
&=Q_3(C_3+Q_3^TD_3^TD_3)+Q_3^T(A_3C_3^TD_3)\\
&=Q_3B+Q_3^TB'
\end{split}
\]
\noindent where $B=C_3+Q_3^TD_3^TD_3$ and $B'=A_3C_3^TD_3$.
\end{proof}

\begin{thrm}
Assume that $p=4k+1$. Let $\mathcal{C}$ be self-dual. Then, $\ds{ \sum_{i=0}^2Q_i }$ is invertible.
\end{thrm}
\begin{proof}
By the previous result, \[ \left( \sum_{i=0}^2Q_i \right)B+\left( \sum_{i=0}^2Q_i  \right)^TB'=I_p\]
for some $p \times p$ matrices $B$ and $B'$ over $R$. Clearly, $Q_i=a_iI_p+b_iQ+c_iN$ where $Q=Q_p(0,1,0),$ $N=Q_p(0,0,1)$. Now,
\[
\begin{split}
Q_i^T&=(a_iI_p+b_iQ+c_iN)^T\\
&=a_iI_p+b_iQ^T+c_iN^T\\
&=a_iI_p+b_iQ+c_iN\\
&=Q_i
\end{split}
\]
\noindent since $Q=Q^T,$ $N=N^T.$ Therefore,
\[ \left( \sum_{i=0}^2Q_i \right)B+\left( \sum_{i=0}^2Q_i  \right)^TB'=\left( \sum_{i=0}^2Q_i \right)B+\left( \sum_{i=0}^2Q_i  \right)B'=\left( \sum_{i=0}^2Q_i  \right)(B+B')=I_p\]
and $\ds{\sum_{i=0}^2Q_i}$ is invertible.
\end{proof}

In the next result, we consider a specific example of a commutative Frobenius ring of characteristic 2. For the purpose of the next result, we assume that $R$ is  a local ring with a residue class field that contains $2$ elements.

\begin{thrm}
Assume that $p=4k+3$, $R$ be a local ring with a residue class field that contains $2$ elements and assume that $k$ is even. Let $\mathcal{C}$ be a self-dual code over $R$. Then, $\ds{ \sum_{i=0}^2Q_i }$ is invertible.
\end{thrm}
\begin{proof} Let $Q_3=\ds{\sum_{i=0}^2} Q_i$, $a_3=\ds{\sum_{i=0}^2} a_i$, $b_3=\ds{\sum_{i=0}^2} b_i$ and $c_3=\ds{\sum_{i=0}^2} c_i$. Clearly, $Q_3=a_3I_p+b_3Q+c_3N$ (where $Q=Q_p(0,1,0),$ $N=Q_p(0,0,1)$) and $Q_3B+Q_3^TB'=I_p$ for some matrices $B$ and $B'$. Let $J$ be the unique maximal ideal in $R$. It remains to show that $Q_3 \pmod{J}$ is invertible. If $b_3\equiv c_3\pmod{J}$ then
\[Q_3^T\equiv (a_3I_p+b_3Q+b_3N)^T\equiv a_3I_p+b_3Q^T+b_3N^T\equiv a_3I_p+b_3N+b_3Q\equiv Q_3\pmod{J}\]
since $Q=N^T$. Therefore, \[
Q_3(B+B')\equiv Q_3B+Q_3^TB'\equiv I_p\pmod{J}.
\]
and $Q_3 \pmod{J}$ is invertible.\\

If $b_3\not\equiv c_3\pmod{J}$ then $b_3+ c_3\equiv 1 \pmod{J}$ and
\[
(\underbrace{1,\ldots,1}_p)Q_3^T=
(\underbrace{1,\ldots,1}_p)Q_3
\equiv (\underbrace{a_3+b_3+c_3,\ldots,a_3+b_3+c_3}_p)
\equiv (a_3+1)(\underbrace{1,\ldots,1}_p)\pmod{J}.
\]
Thus
\[
(\underbrace{1,\ldots,1}_p)Q_3B+(\underbrace{1,\ldots,1}_p)Q_3^TB'= (\underbrace{1,\ldots,1}_p)I_p,
\]

\[
(a_3+1)(\underbrace{1,\ldots,1}_p)(B+B')\equiv (a_3+1)(\underbrace{1,\ldots,1}_p)B+(a_3+1)(\underbrace{1,\ldots,1}_p)B'\equiv (\underbrace{1,\ldots,1}_p)\pmod{J}
\]
and
\[
(a_3+1)(\underbrace{1,\ldots,1}_p)(B+B')(\underbrace{1,\ldots,1}_p)^T\equiv (\underbrace{1,\ldots,1}_p)(\underbrace{1,\ldots,1}_p)^T
\equiv 1\pmod{J}.
\]
So $a_3+1$ is invertible  by modulo ideal $J$ and $a_3\equiv 0 \pmod{J}.$
Thus $Q_3\equiv Q\pmod{J}$ or $Q_3\equiv N\pmod{J}$ and $Q^2= N^2= I_p$ since $k$ is even and $Q^2=N^2=I_p+kQ+kN$. Thus $Q_3 \pmod{J}$ is invertible.
\end{proof}

\section{Numerical results}

\noindent In this section, we construct new self-dual codes of length $66$ and $68$ via certain extensions, neighbours and sequences of neighbours. Initially,
we consider the above construction when $p=5$ over $\FF_2+u\FF_2$. We construct an extremal self-dual code (type I) of length $60$ (described in Table \ref{T1}). From
this code, we construct an extremal self-dual code (type I) of length $64$ via an $\FF_2+u\FF_2$ extension (Table \ref{T2}). Next, we find a new self-dual code of length $66$ by an
$\FF_2$ extension of the previously constructed self-dual code of length $64$ (Table \ref{T3}). Finally, we find new self-dual codes of length $68$ via an $\FF_2+u\FF_2$ extension of the previously constructed self-dual code of length $64$ and sequences of neighbours of this code (Tables \ref{T4}, \ref{T5}, \ref{T6}, \ref{T7} and \ref{T8}).  Magma (\cite{magma}) was used to construct all of the codes throughout this section.  \\

\noindent The possible weight enumerators for a self-dual Type I $\left[ 60,30,12\right]$-code is given in \cites{conway,binary} as:
\begin{eqnarray*}
W_{60,1} &=&1+3451y^{12}+24128y^{14}+336081y^{16}+\cdots , \\
W_{60,2} &=&1+\left(2555+64\beta \right) y^{12}+\left( 33600-384\beta
\right) y^{14}+\cdots ,0\leq \beta \leq 10.
\end{eqnarray*}

Extremal singly even self-dual codes with weight enumerator $W_{60,1}$ and $W_{60,2}$ are known (\cite{har60}) for $\beta \in \{0,1,\ldots,8,10\}$.

\bigskip

\noindent To begin with, we construct the following code:

\begin{table}[h!]\caption{Self-dual codes of length $60$ (codes over $\FF_2+u\FF_2$ when $p=5$)}\label{T1}
\begin{center}\scalebox{0.8}{
\begin{tabular}{|c|c|c|c|c|c|c|c|c|}
  \hline
  $\mathcal{C}_i$ & $(a_1,b_1,c_1)$ & $(a_2,b_2,c_2)$ & $(a_3,b_3,c_3)$ & $v_1$ & $v_2$ & $v_3$ & $Aut(\mathcal{C}_i)$ & $\beta$ \\ \hline \hline
  $1$ & $(u,u,u)$ & $(u,u,1)$ & $(1,u,0)$ & $(u,u,u,u,0)$ & $(u,0,0,u,1)$ & $(u,u+1,u+1,u,0)$ & $2^3 \cdot 3 \cdot 5$ & $0$ \\
  \hline
\end{tabular}}
\end{center}
\end{table}

The possible weight enumerators for a self-dual Type I $\left[ 64,32,12\right]$-code are given in \cite{conway,binary} as:
\begin{eqnarray*}
W_{64,1} &=&1+\left( 1312+16\beta \right) y^{12}+\left( 22016-64\beta
\right) y^{14}+\cdots ,14\leq \beta \leq 284, \\
W_{64,2} &=&1+\left( 1312+16\beta \right) y^{12}+\left( 23040-64\beta
\right) y^{14}+\cdots ,0\leq \beta \leq 277.
\end{eqnarray*}%

Extremal singly even self-dual codes with weight enumerators $W_{64,1}$ are known (\cite{anev,YA,HANKEL})

\[
\beta \in \left\{\begin{array}{c}
14, 16, 18, 19,20, 22, 24, 25, 26, 28, 29, 30, 32, 34, \\
35, 36, 38, 39, 44, 46, 49, 53, 54, 58, 59, 60, 64, 74
\end{array}\right\}
\]

and extremal singly even self-dual codes with weight enumerator $W_{64,2}$ are known for

\[
\beta \in
\left\{\begin{array}{c}
0, . . . ,40, 41, 42, 44, 45, 46, 47, 48, 49, 50, 51, 52, 54, 55, 56, 57,  \\
58, 60, 62, 64, 69, 72, 80, 88, 96, 104, 108, 112, 114, 118, 120, 184
\end{array}\right\} \setminus \{31, 39\}.
\]

\bigskip

The weight enumerators of an extremal self-dual code of length $66$ is given in \cite{binary} as follows:
\begin{eqnarray*}
W_{66,1} &=&1+(858+8\beta )y^{12}+(18678-24\beta )y^{14}+\cdots \text{ \
where }0\leq \beta \leq 778, \\
W_{66,2} &=&1+1690y^{12}+7990y^{14}+\cdots\;\;\text{and }  \\
W_{66,3} &=&1+(858+8\beta )y^{12}+(18166-24\beta )y^{14}+\cdots
\text{ where }14\leq \beta \leq 756.
\end{eqnarray*}%

Together with the codes recently obtained in \cite{anev} and the ones from \cite{Karadeniz2}, \cite{Kaya} and \cite{QR}, extremal singly even self-dual codes with weight enumerator $W_{66,1}$ are known for
$$\beta \in \{0,1,2,3,5,6, \dots, 94, 100, 101, 115\}$$
and extremal singly even self-dual codes with weight enumerator $W_{66,3}$ are known for
$$\beta  \in \{22,23, \dots, 92\} \setminus \{89, 91\}. $$

\bigskip

\noindent The known weight enumerators of a self-dual $[68,34,12]_{I}$-code are as follows (\cite{buyuklieva,harada}):
\begin{eqnarray*} W_{68,1}&=&1+(442+4\beta)y^{12}+(10864-8\beta)y^{14}+ \dots \\
W_{68,2}&=&1+(442+4\beta)y^{12}+(14960-8\beta-256\gamma)y^{14} +\dots \end{eqnarray*}
\noindent where $0\leq \gamma \leq 9$. Codes have been obtained for $W_{68,2}$ when (\cite{GNP})
\[
\begin{split}
\gamma &=2,\ \beta \in \{2m|m=29,\dots ,100,103,104\}; \;\text{or} \; \beta \in\{2m+1|m=32,\dots ,81,84,85,86\}; \\
\gamma &=3,\ \beta \in \{2m|m=39,\dots ,92,94,95,97,98,101,102\};\; \text{or}\;\\
& \qquad \beta \in \{2m+1|m=38,40,43,\dots ,77,79,80,81,83,87,88,89,96\}; \\
\gamma &=4,\ \beta \in \{2m|m=43,46,\dots ,58,60,\dots ,93,97,98,100\};\text{or} \\
& \qquad \beta \in \{2m+1|m=48,\dots ,55,57,58,60,61,62,64,68,\ldots,72,74,78,79,80,83,84,85,89,95\}; \\
\gamma &=5\text{ with }\beta \in \left\{\text{101,105,109,111,$...$,182,187,189,191,192,193,195,198,200,201,202,211,213}\right\}\\
\gamma  &=6,\ \beta \in \left\{ 131,133,137,\ldots,202,203,206,207,210 \right\} ;\\
\gamma  &=7,\ \beta \in \left\{ 7m\left\vert m=14,\ldots
22,28,\ldots,39,42\right. \right\} \text{ or }\beta \in \left\{ 155,\ldots,199 \right\} ;\\
\gamma  &=8,\ \beta \in \left\{ 180,\ldots,221 \right\} ;\\
\gamma  &=9,\ \beta \in \left\{186,\ldots,226,228,230 \right\} ;\\
\end{split}
\]

\noindent Applying Theorem \ref{extension} over $\FF_2$ and $\FF_2+u\FF_2$ (to the code constructed in Table \ref{T1}), we
construct self-dual codes of lengths $64$, $66$ and $68$ (Tables \ref{T2}, \ref{T3} and \ref{T4}). We replace $3$ with $1+u$
to save space.\\

\begin{table}[h!]
\caption{Self-dual codes of length $64$ from $\FF_2+u\FF_2$ extensions of codes from Table \ref{T2} }\label{T2}
\begin{center}\scalebox{0.8}{
\begin{tabular}{|c|c|c|c|c|c|c|}
  \hline
  $\mathcal{D}_i$ & $\mathcal{C}_i$ & $c$ & $X$ & $W_{64,i}$ & $\beta$ & $Aut(\mathcal{D}_i)$  \\ \hline \hline
  $1$ & $1$ & $3$ & $(uu0u3030u330301013u1u1100u1311)$ & $1$ & $14$ & $2^2$  \\ \hline
  \hline
\end{tabular}}
\end{center}
\end{table}

\begin{table}[h!]
\caption{Self-dual codes of length $66$ from $\FF_2$ extensions of codes from Table \ref{T3} where $x_i=0$ for $1 \leq i \leq 33$. }\label{T3}
\begin{center}\scalebox{0.8}{
\begin{tabular}{|c|c|c|c|c|c|c|}
  \hline
  $\mathcal{E}_i$ & $\mathcal{D}_i$ & $c$ & $X$ & $W_{66,i}$ & $\beta$ & $Aut(\mathcal{E}_i)$  \\ \hline \hline
  $1$ & $1$ & $1$ & $(00111100110110011001111001101011)$ & $3$ & $\mathbf{21}$ & $1$  \\ \hline
  \hline
\end{tabular}}
\end{center}
\end{table}

\begin{table}[h!]
\caption{Self-dual codes of length $68$ ($W_{68,2}$) from $\FF_2+u\FF_2$ extensions of codes from Table \ref{T2} }\label{T4}
\begin{center}\scalebox{0.8}{
\begin{tabular}{|c|c|c|c|c|c|c|}
  \hline
  $\mathcal{F}_i$ & $\mathcal{D}_i$ & $c$ & $X$ & $\alpha$ & $\beta$ & $Aut(\mathcal{F}_i)$  \\ \hline \hline
  $1$ & $1$ & $1+u$ & $(0uu01u130130000031100u1u331030u0)$ & $2$ & $67$ & $2$  \\ \hline
  \hline
\end{tabular}}
\end{center}
\end{table}

\noindent Let $\mathcal{N}_{(0)}=\mathcal{F}_1$. Applying the $k^{th}$-range neighbour formula (in section 2), we obtain:

\begin{table}[h]\caption{$i^{th}$ neighbour of  $\mathcal{N}_{(0)}$}\label{neighbors1}\label{T5}
\begin{center}\scalebox{0.8}{
\begin{tabular}{|c|c|ccc|}
\hline
$i$ & $\mathcal{N}_{(i+1)}$   & $x_i$  & $\gamma$ & $\beta$  \\ \hline \hline
$0$ & $\mathcal{N}_{(1)}$ &  $(1010001001111100101010100100000001)$  & $3$ & $103$   \\ \hline
$1$ & $\mathcal{N}_{(2)}$ &  $(1001010100001111001111100011111110)$  & $4$ & $124$   \\ \hline
$2$ & $\mathcal{N}_{(3)}$ &  $(1111101011111101111010000110110111)$  & $5$ & $134$   \\ \hline
$3$ & $\mathcal{N}_{(4)}$ &  $(1010100011100001100011000110010010)$  & $6$ & $149$   \\ \hline
$4$ & $\mathcal{N}_{(5)}$ &  $(0010101000110001011010101011010110)$  & $6$ & $133$   \\ \hline
$5$ & $\mathcal{N}_{(6)}$ &  $(0000001001000111101111000000101110)$  & $\textbf{7}$ & $\textbf{145}$   \\ \hline
$6$ & $\mathcal{N}_{(7)}$ &  $(1101111101111111001111101010111011)$  & $\textbf{8}$ & $\textbf{161}$   \\ \hline
$7$ & $\mathcal{N}_{(8)}$ &  $(1001000001100010000111100000110010)$  & $\textbf{8}$ & $\textbf{153}$   \\ \hline
$8$ & $\mathcal{N}_{(9)}$ &  $(0010111011010011100001110000101111)$  & $\textbf{9}$ & $\textbf{177}$   \\ \hline
\end{tabular}}
\end{center}
\end{table}

\noindent We shall now separately consider the neighbours of $\mathcal{N}_{(7)}$, $\mathcal{N}_{(8)}$ and $\mathcal{N}_{(9)}$.\\

\bigskip

\begin{table}[h]\caption{New codes of length 68 as neighbours}\label{neighbors}\label{T6}
\begin{center}\scalebox{0.7}{
\begin{tabular}{|c|c|ccc|c|c|ccc|}
\hline
$\mathcal{N}_{(i)}$ & $\mathcal{M}_{i}$ & $(x_{35},x_{36},...,x_{68})$ & $\gamma$ & $\beta$  &
$\mathcal{N}_{(i)}$ & $\mathcal{M}_{i}$ & $(x_{35},x_{36},...,x_{68})$ & $\gamma$ & $\beta$  \\ \hline
$7$ & $$ & $(1001110100001011001000010110001111)$  & $\textbf{6}$ & $\textbf{135}$   &
$7$ & $$ & $(0110101110011000110111101110111101)$  & $\textbf{7}$ & $\textbf{142}$   \\ \hline
$7$ & $$ & $(1010101111010000011101101110100001)$  & $\textbf{7}$ & $\textbf{144}$   &
$7$ & $$ & $(1010000001001100100011001110010110)$  & $\textbf{7}$ & $\textbf{148}$   \\ \hline
$7$ & $$ & $(1100000100000100000111110100011000)$  & $\textbf{7}$ & $\textbf{150}$   &
$7$ & $$ & $(0000001101101010011100110000101010)$  & $\textbf{7}$ & $\textbf{152}$   \\ \hline
$7$ & $$ & $(1100001010100000101010001010000011)$  & $\textbf{8}$ & $\textbf{156}$   &
$7$ & $$ & $(0111011101011111010001111101111101)$  & $\textbf{8}$ & $\textbf{157}$   \\ \hline
$7$ & $$ & $(1001110111011110111110110100110111)$  & $\textbf{8}$ & $\textbf{158}$   &
$7$ & $$ & $(1100111101110001001101011111111010)$  & $\textbf{8}$ & $\textbf{159}$   \\ \hline
$7$ & $$ & $(0111111111111101111011010001001110)$  & $\textbf{8}$ & $\textbf{160}$   &
$7$ & $$ & $(0000010100011010000011100000110110)$  & $\textbf{8}$ & $\textbf{162}$   \\ \hline
$7$ & $$ & $(1011100110110111110001111010111001)$  & $\textbf{8}$ & $\textbf{163}$   &
$7$ & $$ & $(1000001100011101010001001011100111)$  & $\textbf{8}$ & $\textbf{164}$   \\ \hline
$7$ & $$ & $(0101101010111111100000010110011010)$  & $\textbf{8}$ & $\textbf{165}$   &
$7$ & $$ & $(1100111110111111011000111101101101)$  & $\textbf{8}$ & $\textbf{166}$   \\ \hline
$7$ & $$ & $(0110110011000101101101010000111011)$  & $\textbf{8}$ & $\textbf{167}$   &
$7$ & $$ & $(1110001001011001000010101101101111)$  & $\textbf{8}$ & $\textbf{168}$   \\ \hline
$7$ & $$ & $(0000110001100111100110010110000100)$  & $\textbf{8}$ & $\textbf{169}$   &
$7$ & $$ & $(1101100001010100111111000110010000)$  & $\textbf{8}$ & $\textbf{170}$   \\ \hline
$7$ & $$ & $(0100111101011101000000001111011110)$  & $\textbf{8}$ & $\textbf{171}$   &
$7$ & $$ & $(1101011100101001111000001010101101)$  & $\textbf{8}$ & $\textbf{172}$   \\ \hline
$7$ & $$ & $(0011011111010111110100010011001110)$  & $\textbf{8}$ & $\textbf{173}$   &
$7$ & $$ & $(1000000111111110110000111001110100)$  & $\textbf{8}$ & $\textbf{174}$   \\ \hline
$7$ & $$ & $(1000111010001101101000001010100111)$  & $\textbf{8}$ & $\textbf{175}$   &
$7$ & $$ & $(1011011001110100101000011000010011)$  & $\textbf{8}$ & $\textbf{176}$   \\ \hline
$7$ & $$ & $(1101110100011011100010110101010001)$  & $\textbf{8}$ & $\textbf{177}$   &
$7$ & $$ & $(0000001001111010000101101011000101)$  & $\textbf{8}$ & $\textbf{178}$   \\ \hline
$7$ & $$ & $(1010110111110111000100101010000110)$  & $\textbf{8}$ & $\textbf{179}$   &
    &    &                                         &              &                   \\ \hline
\end{tabular}}
\end{center}
\end{table}
\vspace{-0.1in}

\begin{table}[h]\caption{New codes of length 68 as neighbours}\label{neighbors}\label{T7}
\begin{center}\scalebox{0.7}{
\begin{tabular}{|c|c|ccc|c|c|ccc|}
\hline
$\mathcal{N}_{(i)}$ & $\mathcal{M}_{i}$ & $(x_{35},x_{36},...,x_{68})$  & $\gamma$ & $\beta$  &
$\mathcal{N}_{(i)}$ & $\mathcal{M}_{i}$ & $(x_{35},x_{36},...,x_{68})$  & $\gamma$ & $\beta$  \\ \hline
$8$ & $$ & $(1011100000000100011001011001010000)$ & $\textbf{6}$ & $\textbf{134}$   &
$8$ & $$ & $(0100011011001110010010110000110000)$  & $\textbf{7}$ & $\textbf{146}$   \\ \hline
$8$ & $$ & $(1000010001101000000110110001001100)$  & $\textbf{8}$ & $\textbf{154}$   &
$8$ & $$ & $(0100010111101000010111100101011101)$  & $\textbf{8}$ & $\textbf{155}$   \\ \hline
\end{tabular}}
\end{center}
\end{table}

\begin{table}[h]\caption{New codes of length 68 as neighbours}\label{neighbors}\label{T8}
\begin{center}\scalebox{0.7}{
\begin{tabular}{|c|c|ccc|c|c|ccc|}
\hline
$\mathcal{N}_{(i)}$ & $\mathcal{M}_{i}$ & $(x_{35},x_{36},...,x_{68})$  & $\gamma$ & $\beta$  &
$\mathcal{N}_{(i)}$ & $\mathcal{M}_{i}$ & $(x_{35},x_{36},...,x_{68})$  & $\gamma$ & $\beta$ \\ \hline
$9$ & $$ & $(1011000010111001011111100101111111)$  & $\textbf{9}$ & $\textbf{169}$   &
$9$ & $$ & $(0111011011011100111010101011101011)$  & $\textbf{9}$ & $\textbf{171}$   \\ \hline
$9$ & $$ & $(1010111001101000111110101111110011)$  & $\textbf{9}$ & $\textbf{173}$   &
$9$ & $$ & $(1000100101111111111101111101000011)$  & $\textbf{9}$ & $\textbf{174}$   \\ \hline
$9$ & $$ & $(1001010100111110011111000101100001)$  & $\textbf{9}$ & $\textbf{175}$   &
$9$ & $$ & $(1100110001000010011000011000010100)$  & $\textbf{9}$ & $\textbf{176}$   \\ \hline
$9$ & $$ & $(0000111100010110110000010011101110)$  & $\textbf{9}$ & $\textbf{178}$   &
$9$ & $$ & $(0000111111001110111000111100010001)$  & $\textbf{9}$ & $\textbf{179}$   \\ \hline
$9$ & $$ & $(0010110110000001011001111001010110)$  & $\textbf{9}$ & $\textbf{180}$   &
$9$ & $$ & $(1101100001101011010000110010101111)$  & $\textbf{9}$ & $\textbf{181}$   \\ \hline
$9$ & $$ & $(1000010010001101110110100111100100)$  & $\textbf{9}$ & $\textbf{182}$   &
$9$ & $$ & $(1111010101110110001110101110011011)$  & $\textbf{9}$ & $\textbf{183}$   \\ \hline
$9$ & $$ & $(0101001111100011111010011011111011)$  & $\textbf{9}$ & $\textbf{184}$   &
$9$ & $$ & $(1011000000001100111100001100011001)$  & $\textbf{9}$ & $\textbf{185}$   \\ \hline
\end{tabular}}
\end{center}
\end{table}

\section{Conclusion}

In this work, we introduced a new construction that involved both block circulant matrices and block quadratic residue circulant matrices. We demonstrated the relevance of this new construction by constructing many binary self-dual codes, including new self-dual codes of length  $66$ and $68$.

\begin{itemize}

\item \textbf{Codes of length $66$:} We were able to construct the following extremal binary self-dual codes with new weight enumerators in $W_{66,3}$:
\[
\beta=\{21\}.
\]

\item \textbf{Codes of length $68$:} We were able to construct the following extremal binary self-dual codes with new weight enumerators in $W_{68,2}$:

\begin{equation*}
\begin{split}
(\gamma =6,& \quad \beta =\{134,135\}). \\
(\gamma =7,& \quad \beta =\{142,144,145,146,148,150,152\}). \\
(\gamma =8,& \quad \beta =\{153,154,155,156,157,158,159,160,161,162,163,164,165,166,167,\\
& \quad \qquad \;\;168,169,170,171,172,173,174,175,176,177,178,179\}).\\
(\gamma =9,& \quad \beta =\{169,171,173,174,175,176,177,178,179,180,181,182,183,184,185\}). \\
\end{split}%
\end{equation*}

\end{itemize}

In this paper, we considered $3 \times 3$ blocks of both block circulant matrices and block quadratic residue circulant matrices. A possible direction in the future could be to consider
$n \times n$ blocks of both block circulant matrices and block quadratic residue circulant matrices.

\end{document}